\documentclass[AMA,STIX1COL]{WileyNJD-v2}
\usepackage{moreverb}

\newcommand\BibTeX{{\rmfamily B\kern-.05em \textsc{i\kern-.025em b}\kern-.08em
T\kern-.1667em\lower.7ex\hbox{E}\kern-.125emX}}

\articletype{Article Type}%

\usepackage{tikz}
\usetikzlibrary{calc,arrows,automata}

\newcommand{\R}{\mathbb{R}}
\newcommand{\lam}{\lambda}
\newcommand{\por}[2]{\underset{\underset{#2}{\uparrow}}{#1} }
\newcommand{\hair}[1]{H(#1)}
\newcommand{\Spec}[1]{\textrm{Spec}(#1)}

\begin{document}

\title{The Palindromic Trees}

\author[1]{Tadashi Akagi}

\author[2]{Eduardo A. Canale*}

\authormark{AUTHOR ONE \textsc{et al}}

\address[1]{\orgdiv{Facutlad Polit\'ecnica}, \orgname{Universidad Nacional de Asunci\'on}, \orgaddress{\state{Central}, \country{Paraguay}}}

\address[2]{\orgdiv{Instituto de Matem\'atica y Estad\'istica}, \orgname{Facultad de Ingenier\'ia, Universidad de la Rep\'ublica}, \orgaddress{\state{Montevideo}, \country{Uruguay}}}

\authormark{T. Akagi  \textsc{et al.}}

\corres{*Eduardo A. Canale,\\IMERL, Facultad de Ingenier\'ia,\\ Julio Herrera y Reissig 565, \\ Montevideo, Uruguay.\\ \email{canale@fing.edu.uy}}

\presentaddress{Present address}

\abstract[Abstract]{  
 The family of trees with palindromic  characteristic polynomials is characterized.
  Large families of  graphs with this property are found as well.
  }


\jnlcitation{\cname{%
\author{T. Akagi} and
\author{E. Canale}
} (\cyear{2022}), 
\ctitle{Palindromic Trees}, \cjournal{Journal of Graph Theory}, \cvol{2022;??:?--?}.}

\maketitle

\section{Introduction}

Despite the extensive research in the field of Spectral Graph Theory, as far as we know the characterization of those graphs with palindromic or
antipalindromic characteristic polynomials has passed unnoticed. Maybe because
 those works that tabulate the characteristic polynomials do not list the leading coefficients because it is always 1. Or maybe, like in \cite{Balasubramanian1982}, the characteristic polynomials are shown factorized. 

In this work we notice that adding a pending vertex to each vertex of
a bipartite graph of (odd)even order gives rise to (anti)palindromic one, and that, in case of the trees, the operation is surjective, in the sense that any (anti)palindromic tree is obtained in this way. 
Unfortunately, the general case is more complex and a general condition has not been found yet.

We will first give  some basic definitions and results on Spectral Graph Theory, as well as some direct results concerning this topic in Section~2. Then, in Section~3, we will  characterize the (anti)palindromic trees in terms of adding pending vertices. We will also notice the relation between this operation and symplectic matrices. Finally, in Section~4, we will show how to produce non-tree palindromic graphs with tensor product of graphs, and we will discuss when this technique gives graphs with pending vertices.

\section{Definitions and Previous Results}

Given a graph $G = (V, E)$ with order $n = |V|$, its adjacent matrix  $A_G  \in \R^{n \times n}$ has elements
\[ (A_G)_{ij} = \left\{ \begin{array}{ll}
     1 & ij \in E,\\
     0 & ij \not\in E.
   \end{array} \right. \]
The characteristic polynomial of $G$ is 
$$
\chi_G(\lambda) = \det (\lambda I
- A_G) = a_0 \lambda^n + a_1 \lambda^{n - 1} + \dots  +a_{n-1}\lam+a_n.
$$ 
Clearly, $a_0 = 1$, while, for the other coefficients, we have,  following
Biggs~{\cite{biggs1974algebraic}}, that
\begin{equation}
  \label{formula} (- 1)^i a_i = \sum_{\Lambda \subset G : | \Lambda | = i} (-
  1)^{r (\Lambda)} 2^{s (\Lambda)},
\end{equation}
where $\Lambda$ is a subgraph of $G$ with connected components isomorphic to either
$K_2$ (the connected graph with two vertices) or a cycle, $r (\Lambda)$ is the rank of $\Lambda$, i.e., $i - c
(\Lambda)$ with $c (\Lambda)$ the number of connected components of $\Lambda$ and $s (\Lambda)$ its corank, i.e., the number of connected components isomorphic to a  cycle.
Clearly, if $G$ has no loops, then $a_1=0$. Besides, $|a_2|$ counts the number of edges of $G$. It  was also proved by Sachs~\cite{cvetkovic1980spectra}, that $G$  is bipartite iff  $a_i = 0$  whenever $i$ is  odd.
Moreover, if $G$ is a tree, the formula becomes
\begin{equation} \label{formula_tree}
    a_i = \begin{cases}
   \displaystyle \sum_{\Lambda \subset G : | \Lambda | = i} (-1)^{i/2}, 
    & i  \text{ even, }\\
       0 & i \text{ odd. }
   \end{cases}
\end{equation}
From this equation it is clear that the coefficient $a_i$ is, in absolute value, the number of matching of cardinality $i/2$ of the tree, i.e., the number of $i$-sets of non adjacent edges (see, for instance, Proposition 7.3 in {\cite{biggs1974algebraic}}).
In particular,   $|a_n|$ counts the number of a perfect matchings,  so $|a_n|\leq 1$.

We say that $G$ is 
\begin{itemize}
    \item \emph{palindromic}  if $a_i = a_{n - i}$
for $i = 0, \ldots, n$.
\item \emph{antipalindromic}  if $a_i = - a_{n - i}$ for $i = 0, \ldots, n$.
\item \emph{absolutely palindromic} if  $|a_i|=|a_{n-i}|$ for $i = 0, \ldots, n$.
\end{itemize}
Notice that $G$ is (anti)palindromic iff $\chi_G$ is as well.
Figure~\ref{tab:first palindromic trees} shows the first eight (anti)palindromic  trees.
\begin{proposition}\label{PMonPalitrees}
Any  (anti)palindromic  tree has exactly one perfect matching.
\end{proposition}
\begin{proof}
Since $|a_n|$ counts the number of perfect matching we need to prove that $|a_n|=1$, which is  true because the tree is (anti)palindromic, so $|a_n|=|a_0|=1$. 
\end{proof}
In fact, this property is  valid for any graph, but without the uniqueness condition. 
Indeed, the palindromic graph on the right hand  side of Fig.~\ref{fig:CounterExample}  has more than one perfect matching.
\begin{proposition}\label{PMonPaliGraphs}
Any  (anti)palindromic graph has a perfect matching.
\end{proposition}
\begin{proof}
If the graph has no  perfect matching, then all the terms in Eq.~\eqref{formula} for $i=n$ will be even, since the exponent of the 2, is always positive for any $\Lambda$. Thus $a_n$ is even, but by the (anti)palindromic hypothesis  $|a_n| = |a_0| = 1$ odd.
\end{proof}
Therefore we have the following corollary.
\begin{corollary}\label{paltheneven}
Every  (anti)palindromic graph has even order.
\end{corollary}

There are some known algebraic conditions for a polynomial to be (anti)palindromic, like the following:  \emph{ a polynomial $p (\lambda)$ of degree $n$ is palindromic iff 
\begin{equation}\label{palindromic}
    p (\lambda) = \lambda^n p (\lambda^{-1}),
\end{equation}
and antipalindromic iff
\begin{equation}\label{antipalindromic}
    p (\lambda) = -\lambda^n p (\lambda^{-1}).
\end{equation}}
Thus, if $\Spec{p}
= (\lambda_1, \lambda_2, \ldots, \lambda_n)$ with $\lambda_i \leqslant
\lambda_{i + 1}$, is the spectrum of a polynomial, then \emph{$p$ is (anti)palindromic iff} $\Spec{p} = \Spec{p}^{- 1}$ and $a_n=1$ ($a_n=-1$ respectively).

Let us define the \emph{hairing} $\hair{G}$ of a graph $G$, as a new graph obtained from the first one by  adding a pending vertex to each of its vertices. More concretely, given a positive integer $k$, let us call the \emph{$k$-hairing}  of a graph $G=(V,E)$  to the graph $H_k(G)$ with vertices $V\times\{0,1,\dots,k\}$ and edges
$(u,0)\sim (v,0)$ if $u \sim v$ in $G$ and $(u,0)\sim(u,i)$ for all $i=1,\dots,k$.
Therefore, the hairing of a graph is its $1$-hairing.

\begin{proposition}
A graph  is bipartite iff  its $k$-hairing  is as well.
\end{proposition}\label{bipartite}
\begin{proof} It follows from the  fact that  both graphs  share the same cycles.
\end{proof}
 \newcommand{\parametros}{}
 
    \newcommand{\Ti}[2]{\begin{tikzpicture}
      [scale=#1,place/.style={circle,draw=black,thick,fill=black, inner sep=0pt,minimum size=#2mm}]

    \foreach \x in {1,2}{
      \node (\x) at ( \x, 0) [place] {};
    };
      	  \draw (1)--(2);
    \end{tikzpicture}    }
\newcommand{\Tii}[2]{         
\begin{tikzpicture}
      [scale=#1,place/.style={circle,draw=black,thick,fill=black, inner sep=0pt,minimum size=#2mm}]

\foreach \x in {2,3}
\node (\x) at ( \x, 0) [place] {};
\node (1) at ( 2, 1) [place] {};
\node (4) at ( 3, 1) [place] {};

      	  \draw (1)--(2)--(3)--(4);
    \end{tikzpicture}   }       
\newcommand{\Tiii}[2]{
\begin{tikzpicture}
        [scale=#1,place/.style={circle,draw=black,thick,fill=black, inner sep=0pt,minimum size=#2mm}]

\foreach \x in {2,3,4}{
\node (\x) at ( \x, 0) [place] {};};
\node (1) at ( 2, 1) [place] {};
\node (5) at ( 4, 1) [place] {};
    \node (6) at ( 3, 1) [place] {};
      	  \draw (1)--(2)--(3)--(4)--(5) (3)--(6);
    \end{tikzpicture}         }     

\newcommand{\Tiv}[2]{
\begin{tikzpicture}
        [scale=#1,place/.style={circle,draw=black,thick,fill=black, inner sep=0pt,minimum size=#2mm}]
    \foreach \x in {2,3,4,5}
\node (\x) at ( \x, 0) [place] {};   \node (1) at ( 2, 1) [place] {};  
\node (6) at ( 3, 1) [place] {};        
    \foreach \x/\y in {7/4,8/5}
       \node (\x) at ( \y, 1) [place] {};
    \draw (1)--(2)--(3)--(4)--(5) (3)--(6) (3)--(7)--(8);
    \end{tikzpicture}   }   
    
\newcommand{\Tv}[2]{
\begin{tikzpicture}
        [scale=#1,place/.style={circle,draw=black,thick,fill=black, inner sep=0pt,minimum size=#2mm}]

    \foreach \x in {2,3,4,5}{
      \node (\x) at ( \x, 0) [place] {};    };
    \node (1) at ( 2, 1) [place] {};
    \node (6) at ( 5, 1) [place] {};
    \foreach \x/\y in {7/3,8/4}
       \node (\x) at ( \y, 1) [place] {};
    \draw (1)--(2)--(3)--(4)--(5)--(6) (3)--(7) (4)--(8);
    \end{tikzpicture}        }     

\newcommand{\Tvi}[2]{
\begin{tikzpicture}
       [scale=#1,place/.style={circle,draw=black,thick,fill=black, inner sep=0pt,minimum size=#2mm}]

    \foreach \x in {1,2,3,4,5}{
      \node (\x) at ( \x, 0) [place] {};    };
    \foreach \x/\y in {7/1,8/2,9/4,10/5,11/3}
       \node (\x) at ( \y, 1) [place] {};
    \draw (1)--(2)--(3)--(4)--(5)  (3)--(8)--(7) (3)--(9)--(10) (3)--(11);
    \end{tikzpicture} }    

\newcommand{\Tvii}[2]{
\begin{tikzpicture}
       [scale=#1,place/.style={circle,draw=black,thick,fill=black, inner sep=0pt,minimum size=#2mm}]

    \foreach \x in {1,2,3,4,5}{
      \node (\x) at ( \x, 0) [place] {};    };
      \node (6) at ( 5, 1) [place] {};      
    \foreach \x/\y in {7/1,8/2,9/3,10/4}
       \node (\x) at ( \y, 1) [place] {};
    \draw (1)--(2)--(3)--(4)--(5)--(6)  (3)--(8)--(7) (3)--(9) (4)--(10);
    \end{tikzpicture}      }     

\newcommand{\Tviii}[2]{
\begin{tikzpicture}
        [scale=#1,place/.style={circle,draw=black,thick,fill=black, inner sep=0pt,minimum size=#2mm}]

    \foreach \x in {2,3,4,5,6}{
      \node (\x) at ( \x, 0) [place] {};    };
    \node (1) at ( 2, 1)[place] {};  
    \node (7) at ( 6, 1)[place] {};
    \foreach \x/\y in {8/3,9/4,10/5}
       \node (\x) at ( \y, 1) [place] {};
    \draw (1)--(2)--(3)--(4)--(5)--(6)--(7) (3)--(8) (4)--(9) (5)--(10);
    \end{tikzpicture}        }

\begin{table}[]
\caption{First eight (anti)palindromic  trees taken from \cite{Mowshowitz}.}
\centering
\begin{tabular}{llll}
\toprule
Tree & Coefficients & Tree & Coefficients \\
& $a_0, a_2, a_4, a_6, a_8, a_{10}$ &&  $a_0, a_2, a_4, a_6, a_8, a_{10}$\\\midrule
   \Ti{0.5}{1.}      &  $1,-1$   &  \Tv{0.5}{1.} & $1, -7, 13, -7, 1$\\
   \Tii{0.5}{1.}  & $1,-3, 1$ & \Tvi{0.5}{1.}  & $1, -9, 22, -22, 9, -1$\\
   \Tiii{0.5}{1.}  & $1, -5, 5, -1$ & \Tvii{0.5}{1.}  & $1, -9, 24, -24, 9, -1$ \\
    \Tiv{0.5}{1.}  & $1, -7, 12, -7, 1$ & \Tviii{0.5}{1.} & $1, -9, 25, -25, 9, -1$\\
    \bottomrule

    \end{tabular}    
    
    \label{tab:first palindromic trees}
\end{table}

\section{Main Results}

\subsection{Hairing a graph}

From Theorem 2.13 in {\cite{cvetkovic1980spectra}}, we known  that if $G$ is a graph with order $n$, then 
\begin{equation*}
  \chi_{H_k(G)} (\lambda) = \lambda^{kn} \chi_G  \left( \lambda -
  \frac{k}{\lambda} \right) .
\end{equation*}
Therefore,  for $k=1$ we obtain:
\begin{equation}\label{Cvec}
  \chi_{H (G)} (\lambda) = \lambda^n \chi_G  \left( \lambda -
  \frac{1}{\lambda} \right) .
\end{equation}
From this equation we have that 
\begin{equation}
    \label{chiH(G)} \chi_{H (G)} (\lambda) = \lambda^{2 n}  (- 1)^n \chi_{H
    (G)}  \left(\frac{-1}\lambda\right) .
  \end{equation}
  Indeed,
  \begin{align*}
  \lambda^{2 n}  (- 1)^n \chi_{H(G)}  \left(-\lam^{-1}\right)  \por{=}{by \eqref{Cvec}} 
\lambda^{2 n}  (- 1)^n \left(-\lam^{-1}\right)^n \chi_G  \left( -\lam^{-1} -
  \frac{1}{-\lam^{-1}} \right)
  =\lambda^{ n}    \chi_G  
\left(  \lambda-\frac{1}\lambda \right) \por{=}{by \eqref{Cvec}} \chi_{H(G)}  (\lambda).    
  \end{align*}
  From  Eq.~\eqref{chiH(G)} we can find a relation between symmetric coefficients.
\begin{lemma}
If $\chi_{H(G)}(\lam)= \lam^{2n}+a_1 \lam^{2n-1}+\cdots + a_{2n}$, then
$$
a_i = (-1)^{n+i} a_{2n-i} \qquad \forall i=0,\ldots,2n.
$$
\end{lemma}
\begin{proof}
Indeed, plugging  the develop of $\chi_{H(G)}(\lam)$ into Eq.~\eqref{chiH(G)} we have:
\begin{align*}
\sum a_i \lam^{2n-i} &=\lam^{2n}(-1)^n \sum a_i (-\lam^{-1})^{2n-i} =\lam^{2n}(-1)^n \sum a_i (-1)^{2n-i}\lam^{-2n+i}
 = \sum a_i (-1)^{n+2n-i}\lam^{i}
\por{=}{j=2n-i}\sum a_{2n-j} (-1)^{n+j}\lam^{2n-j}.
\end{align*}
\end{proof}
From this lemma we have that $H(G)$ is palindromic iff $n+i$ is even for all $i$ or $a_i=a_{2n-i}=0$. However, since $a_0 = 1$ then $n$ should be even, and $a_i = a_{2n-i}=0$ for odd $i$, i.e., $H(G)$ must be bipartite. 
By the same argument, $H(G)$ is antipalindromic iff $n$ is odd and $H(G)$ bipartite.
Altogether we have the following theorem.

\begin{theorem}
The hairing of a graph is always absolutely palindromic. 
While it is  palindromic (respectively antipalindromic) iff it is bipartite and has even order (respectively odd order). \qed
\end{theorem}
Since any tree is bipartite, we have,

\begin{corollary}
  \label{Corolario2} The hairing of a  tree is  (anti)palindromic  iff it has (odd)even order. 
\end{corollary}
All the previous results in the present section are  based on Eq.~\eqref{chiH(G)} derived from Theorem 2.13 in {\cite{cvetkovic1980spectra}}. Let us give a  direct proof of that equation based on ideas from symplectic matrices.

\subsubsection{Relationship with Symplectic matrices}

We can see the adjacent matrix of $H (G)$ as a kind of symplectic one,
indeed, if $A$ and $A'$ are the adjacent matrices  of $G$ and $H (G)$, respectively then
\[ A'  = \left( \begin{array}{cc}
     A & I\\
     I & 0
   \end{array} \right), \]
with a proper sort of the vertices. Clearly, matrix $A'$ is invertible, and it
has determinant equal to $(- 1)^n$, since we can transform matrix $A'$ into the
matrix
\[ \left( \begin{array}{cc}
     I & A\\
     0 & I
   \end{array} \right), \]
by $n$ column transpositions. Therefore, $A'$ is not symplectic when $n$ is
odd, but it verifies the ``quasisymplectic'' equation $A'JA' = - J$ where $J$ is
the {\emph{canonical matrix}}
\[ J = \left( \begin{array}{cc}
     0 & I\\
     - I & 0
   \end{array} \right) . \]
With these equalities in mind, we can proceed like in Appendix 1 titled ``The characteristic polynomial of a symplectic matrix is palindromic'' in  pag. 300 of
{\cite{zee2016group}}  to prove  Eq.~\eqref{chiH(G)}. 
  First, let us notice that the inverse of $A'$ is
  \begin{equation}\label{A'^-1}
   A'^{- 1} = JA'J,
  \end{equation}
  since $A'JA'J = - JJ
  = I$. Now, let us  compute $\chi_{H(G)} (\lambda)$ in the following way:
  \begin{align*}
    \chi_{H(G)} (\lambda) & = \det (\lambda I - A') 
    = \lambda^{2 n} \det (I - \lambda^{- 1} A') 
    = \lambda^{2 n} \det (A') \det (A'^{- 1} - \lambda^{- 1} I)
     \por{=}{\eqref{A'^-1}} \lambda^{2 n}  (- 1)^n \det (JA'J + \lambda^{- 1} JJ) \\
    &= \lambda^{2 n}  (- 1)^n \det (J) \det (A' + \lambda^{- 1} I) \det (J)
     \hspace{-.2cm}\por{=}{\det(J)^2=(\pm)^2=1}\hspace{-.2cm} \lambda^{2 n}  (- 1)^n  \det (A' -(- \lambda^{- 1}) I) =
    \lambda^{2 n}  (- 1)^n  \det (-((- \lambda^{- 1}) I-A')) \\
    &= \lambda^{2 n} 
    (- 1)^n  (- 1)^{2 n} \chi_{H(G)}  (- \lambda^{- 1})
    = \lambda^{2 n}  (- 1)^n  \chi_{H(G)}  (- \lambda^{- 1}),
  \end{align*}
    as we wanted to prove.

\subsection{Characterization of (anti)palindromic trees}

Now we will prove that the only (anti)palindromic trees are those coming from hairing other trees.

\begin{theorem}\label{main_theorem}
  A tree is palindromic (respectively antipalindromic) iff it is the hairing of a tree with even (resp. odd) order.
\end{theorem}

\begin{proof}
By Corollary~\ref{Corolario2} we need to prove that if $T=(V,E)$ is a(n) (anti)palindromic tree, then it is the hairing of a tree. 
By Proposition~\ref{PMonPalitrees} such a tree $T$ has a unique  perfect matching $M = \{v_1 v_{- 1}, \ldots, v_n v_{- n}
  \}$, i.e. $v_i v_{- i} \in E$ and $V = V_+ \cup V_-$ with $V_+ =
  \{v_1, \ldots, v_n \}$ and $V_- = \{v_{- 1}, \ldots, v_{- n} \}$.

Let us suppose, for the sake of contradiction, that $T$ is not the hairing of another tree. Then, we will prove that the number of matchings of cardinality $n - 1$ is greater than the number of edges of $T$, so $|a_2 | > |a_{2n - 2} |$ contradicting the (anti)palindromic property.

  First, let us prove that $|a_2 | \geq |a_{2n - 2} |$, by showing a different
  $(n - 1)$-matching for each edge of the graph. Indeed, for each edge $e$ in
  $M$ the set $M - e$ is a $(n - 1)$-matching. On the other hand, if $e$ is an
  edge not in $M$, then it is adjacent with exactly two edges $f$ and $g$ in $M$, so
  $M - f - g + e$ is a $(n - 1)$-matching, as well. Clearly, all the matchings
  obtained in these  two ways are different.
  
  Now, we will prove the strictness of the inequality, under the hypothesis
  that $T$ is not a hairing of a tree.  Notice that if $T$ is the hairing of a tree, then $M$ should consist exactly of the edges incident with vertices of degree 1 of that hairing. Therefore,  if $T$ is not a hairing,  there will be only two possibilities: 
  \begin{enumerate}
\item  either there is an edge not in $M$ but  with its ends in $V_+$ and $V_-$, or
\item there are edges between vertices in $V_+$ and edges between  vertices in $V_-$ at the same time.
\end{enumerate}
In the first case, we can suppose w.l.o.g. that $v_1 v_{- 2}$ and $v_2 v_3$ are  edges of $T$, see left hand side of  Fig.~\ref{fig:T2}. 
In the second case, by the connectivity of $T$  we can suppose w.l.o.g. that both $v_1 v_2$ and $v_{-2} v_{-3}$ are edges of $T$, illustrated in the right hand side of  Fig.~\ref{fig:T2}.
In  case (1) , the set $M - v_1 v_{- 1} + v_1 v_{- 2} - v_{- 2} v_2 + v_2 v_3 - v_3 v_{- 3}$ is a \mbox{$(n - 1)$-matching} different from those constructed before.
For case (2),  the set 
$M - v_1 v_{- 1} - v_2 v_{- 2} - v_3 v_{- 3} + v_1 v_2 + v_{- 2} v_{- 3}$ is a 
$(n - 1)$-matching different from those constructed before as well, thus the inequality is strict. 
 \end{proof}
  \newcommand{\matching}[1]{
  \begin{tikzpicture}
      [scale=0.8,place/.style={circle,draw=black,thick,fill=black, inner sep=0pt,minimum size=1mm}]

    \foreach \x in {1,2,3}{
      \node (\x) at ( 0, 4-\x) [place] {};
      \node  at ( 0, 4-\x) [left] {$v_{ \x}$};
      \node (\x') at ( 4, 4-\x) [place] {};
      \node  at ( 4, 4-\x) [right] {$v_{  -\x}$};
    };
      \node (n) at ( 0, -0.2) [place] {};
      \node at ( 0, -0.2) [left] {$v_{ n}$};
      \node (n') at ( 4, -0.2) [place] {};
      \node  at ( 4,-0.2) [right] {$v_{  -n}$};
     \node at ( 2, 0.5)  {$\vdots$};
    
    \foreach \x in {1,2,3}
	  \draw (\x) to (\x');
	 \draw (n) to (n'); 
	#1  
    \end{tikzpicture}
   } 
\begin{figure}[ht]
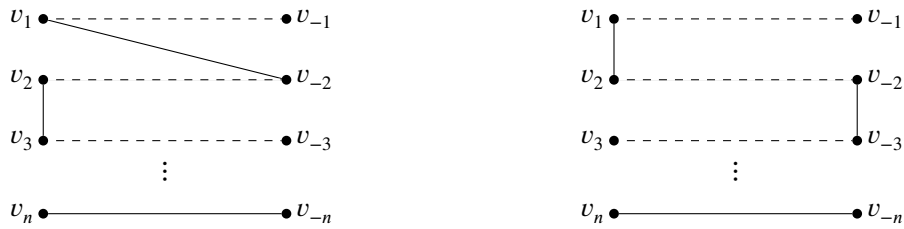

    \centering
        \begin{tabular}{ccc}
    \matching{	
\draw (1)--(2'); 
\draw (2)--(3); 
\draw[draw=white,dashed, thick] (1)--(1'); \draw[draw=white,dashed, thick] (2)--(2'); \draw[draw=white,dashed, very thick] (3)--(3'); 
}         &\mbox{} \hspace{2cm}\mbox{}& 
\matching{	\draw (1) to (2);
\draw (2') to (3');
\draw[draw=white,dashed, thick] (1)--(1'); \draw[draw=white,dashed, thick] (2)--(2'); \draw[draw=white,dashed, very thick] (3)--(3'); 
}
        \end{tabular}
    
    \caption{Cases (1) and (2) of Theorem~\ref{main_theorem}'s proof.}
    \label{fig:T2}
\end{figure}   
 
\begin{corollary}
  Every (anti)palindromic tree has  order multiple of four (plus two).
\end{corollary}
By our computer searches  we think this sentence is valid for any graph.
\begin{conjecture}
   Every (anti)palindromic graph has  order multiple of four (plus two).
\end{conjecture}

\section{General case}

In general there exist antipalindromic and palindromic graphs which are not the hairing of another graph. In Figures~\ref{palindromic n=6} and \ref{palindromic n=8} we show the first of them. Moreover, we can note that there are palindromic graphs with no vertex of degree 1 at all.
In this section we will show how to produce new palindromic graphs through the \emph{tensor product} of two graphs. Given two graphs $G=(V,E)$ and $G'=(V',E')$, the tensor o Kronecker product $G\otimes G'$ of $G$ with $G'$ has vertices $V\times V'$ and edges $\{(u,u')(v,v'):  uv\in E,u'v'\in E'\}$.  Figure~\ref{fig:P_4xP_4} shows the tensor product of the 4-path with  itself. 
Moreover, we will see that this product in general destroys the hairing condition. By considering powers of a graph with itself, we can obtain infinite sequences of palindromic graphs with some properties like the complete  absence of vertices of degree 1. 
Although the tensor product of bipartite graphs produces disconnected ones, taking one connected component  will ultimately be useful to obtain   sequences with a  wider range of orders.

Let us say that a vertex is a \emph{hair} if it has degree 1, and  denote $V^H$ the set of hairs adjacent to a vertex of $V$. 
In an abuse of notation let us denote $G^H$ the hairs of graph $G$. We will say that a graph $G$ is \emph{bald} if it has no hairs, i.e., if $G^H = \emptyset$.

\newcommand{\poligono}[3]{
    \foreach \x in {1,...,#2}{
      \node (\x) at (90+ \x*#3:#1) [place] {};
    };
    \foreach \x in {1,...,#2}
	  \draw (90+ \x*#3:#1) to (90+ \x*#3+#3:#1);
	 
}
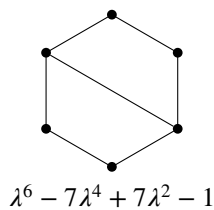
\begin{figure}[ht]
    \centering
    \begin{tabular}{c}
\begin{tikzpicture}
[scale=1,place/.style={circle,draw=black,thick,fill=black, inner sep=0pt,minimum size=1mm}]
\poligono{1}{6}{60};
\draw (1) -- (4); 
\end{tikzpicture}\\
$\lam^6-7\lam^4+7\lam^2-1$ 
    \end{tabular}
    \caption{The unique non hairing antipalindromic graph of order 6.}
    \label{palindromic n=6}
\end{figure}

\newcommand{\escala}{0.9}
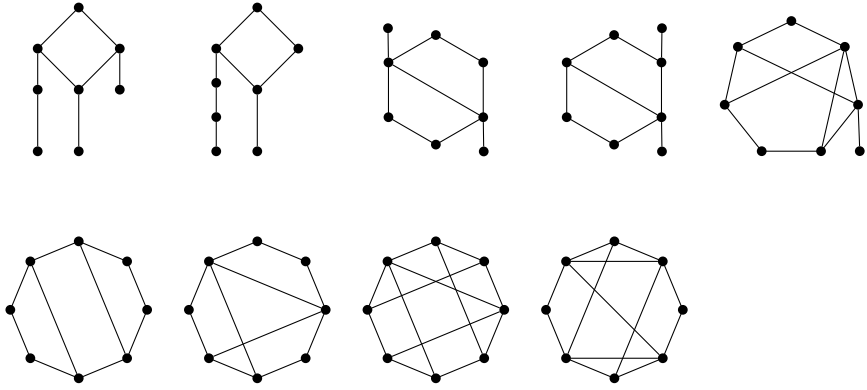
\begin{figure}[ht]
    \centering
    \begin{tabular}{ccccc}
    
\begin{tikzpicture}
[scale=\escala,place/.style={circle,draw=black,thick,fill=black, inner sep=0pt,minimum size=1mm}]
\poligono{0.6}{4}{90};
\node (5) at ( -0.6,-.6) [place] {};
\node (6) at (-.6,-1.5) [place] {};
\node (7) at ( 90+ 2*90:1.5) [place] {};
\node (8) at ( .6, -.6) [place] {};
\draw (1) -- (5)--(6);
\draw (2) -- (7); 
\draw (3) -- (8); 
\end{tikzpicture}& \begin{tikzpicture}
[scale=\escala,place/.style={circle,draw=black,thick,fill=black, inner sep=0pt,minimum size=1mm}]
\poligono{.6}{4}{90};
\node (5) at ( -.6, -0.5) [place] {};
\node (6) at ( -.6, -1.) [place] {};
\node (7) at ( -.6, -1.5) [place] {};
\node (8) at ( 90+ 2*90:1.5) [place] {};
\draw (1) -- (5)--(6)--(7);
\draw (2) -- (8); 
\end{tikzpicture}&    
    
\begin{tikzpicture}
[scale=\escala,place/.style={circle,draw=black,thick,fill=black, inner sep=0pt,minimum size=1mm}]
\poligono{.8}{6}{60};
\draw (1) -- (4); 
\node (7) at ( .7, -.9) [place] {};
\node (8) at ( -.7, .9) [place] {};
\draw (4) -- (7);
\draw (1) -- (8); 
\end{tikzpicture}&
\begin{tikzpicture}
[scale=\escala,place/.style={circle,draw=black,thick,fill=black, inner sep=0pt,minimum size=1mm}]
\poligono{.8}{6}{60};
\draw (1) -- (4); 
\node (7) at ( .7, -.9) [place] {};
\node (8) at (.7, .9) [place] {};
\draw (4) -- (7);
\draw (5) -- (8); 
\end{tikzpicture}&
\begin{tikzpicture}
[scale=\escala,place/.style={circle,draw=black,thick,fill=black, inner sep=0pt,minimum size=1mm}]
\poligono{1}{7}{51.4285};
 \node (8) at (1, -.9) [place] {};
\draw (5) -- (1);
\draw (6) -- (2);
\draw (4) -- (6);

\draw (5) -- (8);
\end{tikzpicture}\\
\\&&\\
\begin{tikzpicture}
[scale=\escala,place/.style={circle,draw=black,thick,fill=black, inner sep=0pt,minimum size=1mm}]
\poligono{1}{8}{45};
 \draw (1) -- (4);
\draw (5) -- (8);
\end{tikzpicture}&
\begin{tikzpicture}
[scale=\escala,place/.style={circle,draw=black,thick,fill=black, inner sep=0pt,minimum size=1mm}]
\poligono{1}{8}{45};
 \draw (1) -- (4);
\draw (3) -- (6);
\draw (6) -- (1);
\end{tikzpicture}&
\begin{tikzpicture}
[scale=\escala,place/.style={circle,draw=black,thick,fill=black, inner sep=0pt,minimum size=1mm}]
\poligono{1}{8}{45};
 \draw (1) -- (4);
\draw (3) -- (6);
\draw (6) -- (1);
\draw (2) -- (7);
\draw (8) -- (5);
\end{tikzpicture}&
\begin{tikzpicture}
[scale=\escala,place/.style={circle,draw=black,thick,fill=black, inner sep=0pt,minimum size=1mm}]
\poligono{1}{8}{45};
 \draw (1) -- (5);
\draw (3) -- (5);
\draw (3) -- (8);
\draw (1) -- (7);
\draw (4) -- (7);
\end{tikzpicture}\\
\\&&
\end{tabular}
     \caption{The nine non hairing palindromic graphs of order 8.}
    \label{palindromic n=8}
\end{figure}

\begin{figure}[ht]
    \centering
    \begin{tabular}{ccc}

\begin{tikzpicture}
[scale=0.7,place/.style={circle,draw=black,thick,fill=black, inner sep=0pt,minimum size=1mm}]
\poligono{1}{5}{72};
\draw (2) -- (4); 
\node (6) at ( 90+ 1*72:1.5) [place] {};
\node (7) at ( 90+ 2*72:1.5) [place] {};
\node (8) at ( 90+ 3*72:1.5) [place] {};
\draw (1) -- (6);
\draw (2) -- (7); 
\draw (3) -- (8); 
\end{tikzpicture}&
\begin{tikzpicture}
[scale=0.7,place/.style={circle,draw=black,thick,fill=black, inner sep=0pt,minimum size=1mm}]
\poligono{1}{7}{51.4285};
 \node (8) at ( 90-2*51.4285:1.5) [place] {};
\draw (5) -- (8);
\draw (4) -- (7)--(3); 
\end{tikzpicture}&
\begin{tikzpicture}
[scale=0.7,place/.style={circle,draw=black,thick,fill=black, inner sep=0pt,minimum size=1mm}]
\poligono{1}{8}{45};
\draw (7)--(2)--(5)--(1) -- (4)--(7); 
\draw (4) -- (8)--(5); 
\end{tikzpicture}\\
{\small$\lam^8-9\lam^6-2\lam^5-18\lam^4+2\lam^3-9\lam^2+1$}&
{\small$\lam^8-10\lam^6-2\lam^5-23\lam^4+2\lam^3-10\lam^2+1$}&
{\small$\lam^8-15\lam^6-12\lam^5+25\lam^4+12\lam^3-15\lam^2+1$}\\
    \end{tabular}
     \centering
    \caption{Three out of 17 absolutely  palindromic graphs of order 8.}
    \label{absolutely palindromic}
\end{figure}

\begin{proposition}\label{prop:KP}
The tensor product of two palindromic or antipalindromic
  graphs is always palindromic. 
\end{proposition}  
\begin{proof}
  
Indeed, the adjacent matrices of the tensor
  product is the tensor product of the adjacent matrices, but the last has
  as spectrum the product of the spectra of its factors, i.e., if $\{\lambda_1, \ldots,
  \lambda_n \}$ and $\{\mu_1, \ldots, \mu_m \}$ are the spectra of matrices $M_1$ and $M_2$, then
  $\{\lambda_i \mu_j : i = 1, \ldots, n, j = 1, \ldots, m\}$ is the spectrum
  of the tensor product of $M_1$ and $M_2$. Since a graph is palindromic or antipalindromic  iff $\Spec{G}:=\Spec{A_G}=\Spec{G}^{-1}$, then if $G_1$ and $G_2$ are palindromic or antipalindromic, we have 
  
  \begin{align*}
    \Spec{G_1 \otimes G_2}^{- 1} & 
    = \{(\lambda \mu)^{- 1} : \lambda \in \Spec{G_1}, \mu \in \Spec{G_2}\}\\
    & = \{\lambda^{- 1} \mu^{- 1} : \lambda \in \Spec{G_1}, \mu \in \Spec{G_2}\}\\
    & = \{\lambda \mu : \lambda^{- 1} \in \Spec{G_1}, \mu^{- 1} \in \Spec{G_2}\}\\
    & = \{\lambda \mu : \lambda \in \Spec{G_1}^{- 1}, \mu \in \Spec{G_2}^{-
    1} \}\\
    & = \{\lambda \mu : \lambda \in \Spec{G_1}, \mu \in \Spec{G_2}\}\\
    & = \Spec{G_1 \otimes G_2}.
  \end{align*}
  Therefore, $G_1\otimes G_2$ is palindromic or antipalindromic.
  In order to discard the antipalindromic option, we need to prove that the   independent term $a_{nm}$ of $\chi_{ G_1 \otimes G_2}$ is $1$ instead of $-1$. 
  But $a_{nm}$  is equal to the evaluation of $\chi_{ G_1 \otimes G_2}$  in 0, as well as  to the product of the opposite of its eigenvalues, i.e.
   $a_{nm}=\chi_{ G_1 \otimes G_2}( 0) =   \prod_{\tau \in \Spec {G_1\otimes G_2} }
  (-\tau)$, so
  
  \begin{align*}
     \chi_{ G_1 \otimes G_2}( 0) & = (- 1)^{|G_1 | |G_2 |} \prod_{\tau \in
    \Spec{G_1 \otimes G_2}} \tau \por{=}{|G_1 |,|G_2 | \text{ even}} \prod_{\lam \in \Spec{G_1}}\prod_{ \mu \in 
    \Spec{G_2}} \lam \mu
    =\prod_{\lam \in \Spec{G_1}}\lam^{|G_2|}
    \prod_{ \mu \in \Spec{G_2}}\mu\\
     &= \left(
    \prod_{\mu \in \Spec{G_2}} \mu \right)^{|G_1 |} \left( \prod_{\lam \in \Spec{G_1}} \lam \right)^{|G_2 |}
     =   
    (\pm 1)^{|G_1 |}  (\pm1)^{|G_2 |} \por{=}{|G_1 |,|G_2 | \text{ even}} 1.
  \end{align*}  
\end{proof}
Naturally, one  wonders if  the converse of this result is true as well as if the graphs obtained in this way  are  different graphs  from those obtained by hairing. The former question is false, as the counterexample in  Figure~\ref{fig:CounterExample} proves, while the latter is answered positively below.

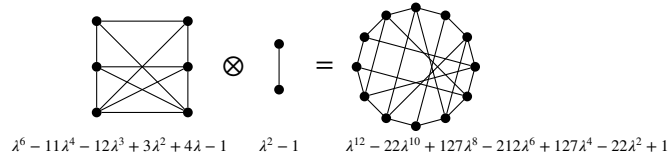
\begin{figure}[ht]
    \centering
    \begin{tikzpicture}
      [scale=.6,place/.style={circle,draw=black,thick,fill=black, inner sep=0pt,minimum size=1mm}]

    \foreach \x in {1,2,3}{
      \node (\x) at  (0, \x) [place] {};
      \node (\x') at (2, \x) [place] {};
    };
    \foreach \x in {1,2,3}
      \draw (\x) --  (\x');
    \foreach \x/\y in {1/2,2/3}{
      \draw (\x) --  (\y);
      \draw (\x') --  (\y');
};
  \foreach \x/\y in {1/2,1/3}{
      \draw (\x) --  (\y');
      \draw (\x') --  (\y);
};
      
\node[scale=0.6] at (.5,0.3) {$\lam^6-11\lam^4-12\lam^3+3\lam^2+4\lam-1$};      
\node[scale=0.6] at (4,0.3) {$\lam^2-1$};      
\node[scale=0.6] at (9,0.3) {$\lam^{12}-22\lam^{10}+127\lam^8-212\lam^6+127\lam^4-22\lam^2+1$};      
      
    \node at (3,2) {$\otimes$};
    \foreach \x in {1,2}
      \node (\x'') at  (4, \x+.5) [place] {};
    \draw (1'') --  (2'');
    \node at (5,2) {$=$};
    \begin{scope}[xshift = 7 cm, yshift = 2 cm]
    \poligono{1.3}{12}{30}
    \draw (1) --  (8)--(3);
    \draw (2) --  (9)--(4)--(1);
    \draw (5) --  (12)--(7)--(10)--(5);
    \draw (6) --  (11);
    \end{scope}

    \end{tikzpicture} 
    \caption{Counterexample to the converse of Proposition~\ref{prop:KP}}
    \label{fig:CounterExample}
\end{figure}

Let us first remember some  properties of the tensor's  product of graphs.
\begin{proposition}\label{Prop:KroenBip}
Let $G$ be the tensor's product  of two connected graphs $G_1$ and $G_2$. Then 
\begin{enumerate}
    \item The degree of a vertex $(u,v)$ of $G$ is the product of the degrees of $u$ and $v$. 
    \item A vertex $(u,v) \in G^H$  iff both $u\in G_1^H$ and $v \in G_2^H$.
    \item  $G$ is connected iff one of $G_i$  is not bipartite.
    \item  $G$ is bipartite iff one of $G_i$ is bipartite.
    \item If both $G_i$ are bipartite, then $G$  has exactly two (bipartite) connected components.  Moreover, if $G_1$ and  $G_2$  have  parts  $( V_1 , W_1)$ and $( V_2 , W_2)$, respectively, then $G$ has one connected component with parts $(V_1 \times V_2, W_1 \times W_2)$ and the other component with parts $(V_1 \times W_2, W_1 \times V_2)$.
    \end{enumerate}   
    \qed
\end{proposition}
Given this result one can ask if the two connected components of the product of two bipartite graphs are palindromic if the original graphs were palindronmic too. The answer is again positive as shown next.
\begin{proposition}\label{product_of_bip_pal}
The two connected bipartite components of the tensor product of two (anti)palindromic bipartite graphs are palindromic.
\end{proposition}
\begin{proof}
Its is enough to prove that these components share the same eigenvalues of the product graph $P$.
Let $ z$  be an eigenvector of $P$ with eigenvalue $\mu$, and $A_C$ the adjacent matrix of one of its connected components $C$. 
Then, it will always happen that $A  z_C  = \mu  z_C $, where $ z_C$ is the part of $ z$ corresponding to the vertices of $C$. 
So $ z_C $ will be an eigenvector of $A_C$ with eigenvalue $\mu$, whenever $ z_C $ is not null. 
In order to prove the non nullity of $ z_C $ we need to go a bit deeper in the form of $ z$. 
The eigenvector $ z$ will have coordinates $z_{uu'}=(x_u y_{u'})$ with $x$ and $y$ eigenvectors of the graphs in the product.
If the graph corresponding to $x$ has parts $(V,W)$ and adjacent matrix
$$
 \left(\begin{array}{cc}0&A_1\\A_2&0\end{array}\right),
$$
then $A_1 x_{W} = \lam x_{V}$ and  $A_2 x_{V} = \lam x_{W}$.
But since the graphs are (anti)palindromic, their eigenvalues are not null, so $\lam \neq 0$ and $x_V$ is null iff so it is $x_W$. 
Therefore, since $x$ is not null,  neither $x_V$  nor $x_W$ are, and there will be an $v_0 \in V$ with $x_{v_0} \neq 0$.
The same argument is valid for $y$, so  if $(V',W')$ are the parts of the other graph of the product, there will be coordinates not null $y_{v'_0} \in V', y_{w'_0}\in W'$. 
Therefore $x_{v_0}y_{v'_0}$ and  $x_{v_0}y_{w'_0}$ are not null coordinates of $z$ corresponding to  the two components of $P$.
\end{proof}
From parts (1) and (2) in Proposition~\ref{Prop:KroenBip}  it is clear the  following relationship between the hairs of two graphs and the hairs of its product.
\begin{proposition}\label{Prop_KroBip}
Let $G$ be the tensor product  of two connected bipartite  graphs $G_1$ and $G_2$ with parts $(V_1 , W_1)$ and  $(V_2 , W_2)$ respectively, then
\begin{enumerate}
    \item  $G^H = G_1^H \times G_2^H $.
     \item  The hairs of the components of $G$ are $$((V_1 \times V_2)\cup (W_1 \times W_2))^H=(V_1^H \times V_2^H)\cup (W_1^H \times W_2^H)$$ and 
      $((V_1 \times W_2)\cup (W_1 \times V_2))^H =(V_1^H \times W_2^H)\cup (W_1^H \times V_2^H)$. 
\end{enumerate}
\qed
\end{proposition}
We are now in position to prove that the product of two hairing graphs is not hairing. 
Let us first identify and count the cardinality of the parts of the hairing of a bipartite graph.

\begin{lemma}\label{Lemma:HairBip}
The hairing of a bipartite graph is bipartite with parts of the same cardinality.
\end{lemma}
\begin{proof}
If $G$ is a bipartite graph with parts $V$ and $W$, then the hairing $H(G)$ has parts $(V\cup W^H, W \cup V^H)$. But  $|W^H|=|W|$ and $|V^H|=|V|$,
therefore  $|V\cup W^H| = |V|+|W^H|=|V^H|+|W|=|V^H\cup W|$.
\end{proof}
As we said, the tensor product of two hairing graphs is, in most of the cases, not a hairing one.
\begin{proposition}\label{nothairin}
The tensor product of two hairing graphs different from $K_2$ is  not a hairing one.
\end{proposition}
\begin{proof}
Since the number of hairs of a hairing graph different from $K_2$ is half the order of the graph, 
if one hairing graph has $2n$ vertices and the other one $2m$, then, their product has $4nm$ vertices but $nm$ hairs which is smaller than half $4nm$.
\end{proof}
Despite this proposition, since the product of bipartite graphs is not connected, it could eventually happen that although the product is not hairing one of its components does. 
However, this is not the case, as we establish next.

\begin{proposition}\label{nohairing}
The tensor product of two bipartite graphs with an amount of hairs at most $\alpha$ times its orders has two connected components with an amount of hairs at most $\alpha^2 $ times its orders. If both graphs are the hairing of a graph, then, none of the two connected components of its tensor product is the hairing graph of another one and  both components have the same order.
\end{proposition}
\begin{proof}
By Proposition~\ref{Prop_KroBip}, and following the notation there, the ratio between the cardinalities of the hairs of each parts of the product with the cardinality of the part is:
\begin{align*}
\frac{|(V_1^H \times V_2^H)\cup (W_1^H \times W_2^H)|}{|(V_1 \times V_2) \cup (W_1 \times W_2)|}&
= \frac{|V_1^H||V_2^H|+|W_1^H||W_2^H|}{|V_1|| V_2|+ |W_1|| W_2|}\leq 
 \frac{\alpha|V_1|\alpha|V_2|+\alpha|W_1|\alpha|W_2|}{|V_1|| V_2|+ |W_1|| W_2|}= \alpha^2,
\end{align*}
for one of the parts, while, by a  similar computation, the other ratio is
\begin{equation*}
\frac{|(V_1^H \times W_2^H)\cup (W_1^H \times V_2^H)|}{|(V_1 \times W_2) \cup (W_1 \times V_2)|}\leq
\frac{\alpha|V_1|\alpha|W_2|+\alpha|W_1|\alpha|V_2|}{|V_1|| W_2|+ |W_1||V_2|} 
= \alpha^2.
\end{equation*}
 If both initial graphs are the hairing of another one, then $\alpha = 1/2$, and the ratio is $1/4$, therefore none of them could be a hairing.
 
In order to prove that both parts have the same cardinality, let us observe that if one graph is the hairing of a bipartite graph with parts $(V_1,W_1)$ and the other the hairing of a bipartite  graph with parts $(V_2, W_2)$, then, the product of the hairing has parts $((V_1\cup W_1^H)\times(V_2\cup W_2^H))\cup ((W_1\cup V_1^H)\times(W_2\cup V_2^H))$ and 
$((V_1\cup W_1^H)\times(W_2\cup V_2^H))\cup ((W_1\cup V_1^H)\times(V_2\cup W_2^H))$.
The first part has cardinality
\begin{align*}
&|((V_1\cup W_1^H)\times(V_2\cup W_2^H))\cup ((W_1\cup V_1^H)\times(W_2\cup V_2^H))|\\
&=(|V_1|+ |W_1|)(|V_2|+|W_2|)+ (|W_1|+ |V_1|)(|W_2|+|V_2|)\\
&=2(|V_1|+ |W_1|)(|V_2|+|W_2|),
\end{align*}
while the second part has cardinality
\begin{align*}
&|((V_1\cup W_1^H)\times(W_2\cup V_2^H))\cup ((W_1\cup V_1^H)\times(V_2\cup W_2^H))|\\
&=(|V_1|+ |W_1|)(|W_2|+|V_2|)+ 
((|W_1|+ |V_1|)(|V_2|+|W_2|)\\
&=2(|V_1|+ |W_1|)(|V_2|+|W_2|),
\end{align*}
which is equal to the first one.
\end{proof}
Let us notice that although  the two connected components mentioned in the proposition have the same order, they need not be isomorphic, since they could even have different amounts of hairs. Indeed, as an example, we can take the product with itself of the hairing of a bipartite graph with parts with different amounts of vertices. If the graph has parts with $a \neq b$ number of vertices, then one component of the product of its hairing with itself will have $a^2+b^2$ hairs while the other one will have $2ab$.

 From these propositions,  we can produce an infinite family of non hairing palindromic graphs. Indeed, from the  Proposition~\ref{product_of_bip_pal} we have that the product of palindromic trees has  two palindromic components of the same order, say $2h$, while from Proposition~\ref{nohairing} none of them is hairing.  Therefore,  by Theorem~\ref{main_theorem} none of them is  a tree either. From this any of these non--trees non--hairing palindromic graphs, we will obtain, by a new tensor product with any other palindromic tree of order $2n$ another  connected non hairing bipartite palindromic graph of order $2h\times 2n /2 = 2 hn$.
 For instance, if we begin by multiplying the  smallest palindromic tree, i.e. $P_4$, with itself, we have $h=4$, and we will obtain a non hairing palindromic graph with order $16n$ for any $n\geq 2$.
 Figure~\ref{fig:P_4xP_4} shows the product of $P_4$ with itself and the two connected non hairing bipartite palindromic non-tree graphs, which, in this case, are isomorphic to the third graph in Figure~\ref{palindromic n=8}.
 From this graph we obtain an infinite family of non haring non-tree palindromic graphs or orders $8n$ for any $n\geq 2$.

 As we showed in the last four graphs in Figure~\ref{palindromic n=8}, there are palindromic \emph{bald} graphs. 
 The first of these four is bipartite, so  we can obtain an infinite family of palindromic bald graphs and order again $8\times 2n /2 = 8n$ for any $n\geq 2$.
 
 Summarizing, we have the following theorem.
\begin{theorem}
  There are palindromic bald graphs, for any  order  multiple of $8$.
\end{theorem}

\begin{figure}[ht]
    \centering
    \begin{tikzpicture}
      [scale=.6,place/.style={circle,draw=black,thick,fill=black, inner sep=0pt,minimum size=1mm}]

    \foreach \x in {1,2,3,4}{
      \node (\x) at (0, \x) [place] {};
      \node (\x') at (2, \x) [place] {};
    };
    \draw (1)--(2)--(3)--(4);
    \draw (1')--(2')--(3')--(4');
    \node  at (1, 2.5)  {$\otimes$};
    \node  at (3, 2.5)  {$=$};
    \node  at (8, 2.5)  {$=$};
    \node  at (13, 2.5)  {$\cup$};
    \begin{scope}[xshift = 3 cm]
    \foreach \x in {1,2,3,4}
        \foreach \y in {1,2,3,4}{
      \node (\x\y) at (\x, \y) [place] {};
    };
    \foreach \x/\z in {1/2,2/3,2/1,3/2,3/4,4/3}
        \foreach \y/\t in {1/2,2/3,2/1,3/2,3/4,4/3}{
      \draw  (\x\y)--(\z, \t) ;
    };
    \end{scope}
    \begin{scope}[xshift = 8 cm]
    \foreach \x/\y in {1/1,1/3,2/2,2/4,3/1,3/3, 4/2,4/4}{
      \node (\x\y) at (\x, \y) [place] {};
    };
    \foreach \x/\y in {11/22,22/13,22/31,22/33,33/24,33/42,33/44,31/42,42/31,42/33,13/24}
      \draw  (\x)--(\y) ;
    \end{scope}
    \begin{scope}[xshift = 13 cm]
    \foreach \x/\y in {1/2,1/4,2/1,2/3,3/2,3/4,4/1,4/3}{
      \node (\x\y) at (\x, \y) [place] {};
    };
    \foreach \x/\y in {12/21,12/23,14/23,21/32,23/32,23/34,32/41,34/43,32/43}
      \draw  (\x)--(\y) ;
    \end{scope}
    \end{tikzpicture} 
    \caption{$P_4\otimes P_4$}
    \label{fig:P_4xP_4}
\end{figure}
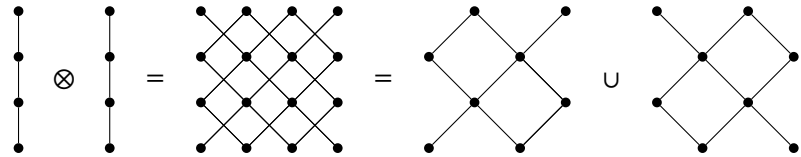
From the 5th and 9th graphs in Figure~\ref{palindromic n=8} we check that there are non bipartite palindromic graphs. Moreover, the 9th graph is bald, so the product of  this graph with itself $k$ times gives  a sequence of palindromic non bipartite bald graphs with orders $8^k$.

By computer search aided by Nauty~\cite{nauty},  we have verified that there are no palindromic graphs of order 10, but 53 antipalindromic ones and 326 absolutely palindroimc non antipalindromic. From those 53, 3 are hairing of the 3 non isomorphic trees of order 5, and 18 are hairing of the other connected non-tree graphs. On the other hand, in the case of order 12, we have computed 368 palindromic triangle free graphs and only 43 absolutely palindromic triangle free graphs which are not palindromic. 
The condition of triangle free seems to reduce the number of absolutely palindromic. For instance, in the case of order 8, there are only 2 free triangles absolutely palindromic out of 21. We summarize our numerical results in Table~\ref{tab:results}.

\begin{table}[h]
\caption{Computer results, where P. =``Palindromic'', A. =``Antipalindromic', $|$P.$|$ = ``Absolutely Palindromic'', T. =``tensor'', P. = ``Product'', H. = ``Hairing''.  \cite[A005142,A001349]{oeis} }
    \label{tab:results}
    \centering
    \begin{tabular}{ccccccccc}
    \toprule
    $n$ & P. & A & $|$P.$|$& Trees & T.P. & H.P. & H.A. & H.  $|$P.$|$ \\\midrule
    2     &  0  & 1 & 1  & 1 & 0 & 0 & 1 & 1 \\
    4     &  1  & 0 & 1  & 1 & 0 & 1 & 0 & 1 \\
    6     &  0  & 4 & 4  & 1 & 0 & 0 & 2 & 2 \\
    8     &  14 & 0 & 35 & 2 & 1 & 5 & 0 & 4\\
    10    & 0   & 53& 326& 3 & 0 & 0 & 5 & 9  \\
    12 &  $\geq368$&0&  $\geq389$& 7 & 6 & 17 & 0  & 22\\
    &\\\bottomrule
    \end{tabular}
    
\end{table}

\section{Conclusions}
We have characterized those (anti)palindromic trees as those made by adding a pending vertex to each vertex of a given graph.
This implies that the structural complexity of a (anti)palindromic tree could be  as rich as the family of trees itself. 

From a computational point of view, one can wonder which advantage gives  the present characterization. 
A quick answer could be that this characterization  allows a recognition in linear time, in contrast with the $O (n^{2.3})$ algorithm to compute the characteristic polynomial based on the Coppersmith-Winograd algorithm.

Unfortunately, despite the tree's case, applying the map $H$ to a graph is not the only way to obtain (anti)palindromic graphs. So we left  the general case open. 
Nevertheless a different way to produce palindromic graphs was introduced, such that a sequence of graphs with no pending vertices is  built.

\bibliography{Palindromic}%







\end{document}